\documentclass[10pt,reqno]{amsart}
\usepackage{latexsym,amsmath,amssymb,amscd}
\usepackage[all]{xy}
\usepackage{enumerate}

\def\today{\ifcase \month \or
	January \or February \or March \or April \or
	May \or June \or July \or August \or
	September \or October \or November \or December \fi
	\space\number\day , \number\year}

\newtheorem{theorem}{Theorem}

\newtheorem{corollary}[theorem]{Corollary}
\newtheorem{definition}[theorem]{Definition}
\newtheorem{example}[theorem]{Example}
\newtheorem{lemma}[theorem]{Lemma}

\newtheorem{proposition}[theorem]{Proposition}

\newtheorem{remark}[theorem]{Remark}

\numberwithin{theorem}{section}
\numberwithin{equation}{section}

\newcommand{\Ad}{{\rm Ad}}

\newcommand{\card}{{\rm card}}
\newcommand{\clgth}{\text{\rm c-length}\,}

\newcommand{\ee}{{\rm e}}

\newcommand{\ind}{{\rm ind}\,}
\newcommand{\Ker}{{\rm Ker}\,}
\newcommand{\lgth}{{\rm length}\,}
\newcommand{\Prim}{{\rm Prim}}

\newcommand{\RRa}{{\rm RR}}

\newcommand{\sa}{{\rm sa}}
\newcommand{\spa}{{\rm span}\,}
\newcommand{\tsr}{{\rm tsr}}
\newcommand{\Tr}{{\rm Tr}\,}

\newcommand{\CC}{{\mathbb C}}

\newcommand{\RR}{{\mathbb R}}
\newcommand{\TT}{{\mathbb T}}

\newcommand{\Ac}{{\mathcal A}}
\newcommand{\Bc}{{\mathcal B}}
\newcommand{\Cc}{{\mathcal C}}
\newcommand{\Ec}{{\mathcal E}}
\newcommand{\Fc}{{\mathcal F}}
\newcommand{\Gc}{{\mathcal G}}
\newcommand{\Hc}{{\mathcal H}}
\newcommand{\Ic}{{\mathcal I}}
\newcommand{\Jc}{{\mathcal J}}

\newcommand{\Lc}{{\mathcal L}}

\newcommand{\Oc}{{\mathcal O}}

\newcommand{\Sc}{{\mathcal S}}
\newcommand{\Qc}{{\mathcal Q}}

\newcommand{\Vc}{{\mathcal V}}
\newcommand{\Xc}{{\mathcal X}}

\newcommand{\ag}{{\mathfrak a}}
\renewcommand{\gg}{{\mathfrak g}}
\newcommand{\hg}{{\mathfrak h}}

\newcommand{\mg}{{\mathfrak m}}
\renewcommand{\ng}{{\mathfrak n}}

\newcommand{\sg}{{\mathfrak s}}
\newcommand{\zg}{{\mathfrak z}}

\newcommand{\NN}{\mathbb N}

\begin{document}

\title{Topological aspects of group $C^*$-algebras}

\author{Ingrid Belti\c t\u a and Daniel Belti\c t\u a}
\address{Institute of Mathematics ``Simion Stoilow'' 
	of the Romanian Academy, 
	P.O. Box 1-764, Bucharest, Romania}
\email{ingrid.beltita@gmail.com, Ingrid.Beltita@imar.ro}
\email{beltita@gmail.com, Daniel.Beltita@imar.ro}
\thanks{This work was supported by a grant of the Romanian National Authority for Scientific Research and
	Innovation, CNCS--UEFISCDI, project number PN-II-RU-TE-2014-4-0370}

\begin{abstract}
We discuss basic topological properties of unitary dual spaces of nilpotent Lie groups, 
using some ideas from operator algebras and their noncommutative dimension theory. 
The general results are illustrated by many examples. 
\end{abstract}
\maketitle


\section{Introduction}

Representation theory of a locally compact group can be developed to a large extent 
using the $C^*$-algebra of that group. 
Typically, the $C^*$-algebras that arise in this way have a rich supply of primitive ideals. 
This explains why an important role in this approach to group representation theory is played 
by the topological aspects of $C^*$-algebras, by which we mean  
properties of the Jacobson topology on the space of primitive ideals of a $C^*$-algebra.   
If $G$ is a locally compact group of type~I, 
the space of primitive ideals of its group $C^*$-algebra $C^*(G)$ is canonically homeomorphic to the unitary dual~$\widehat{G}$, so the topological aspects of $C^*(G)$ amount to the study of topological properties of~$\widehat{G}$.

A central place in this area of $C^*$-algebra theory is held by 
the inverse problems for primitive ideal spaces--- for a given class of $C^*$-algebras, 
one tries to single out  the topological spaces that are homeomorphic to primitive ideal spaces of $C^*$-algebras from that class. 
An important topic is the study of
Lie group $C^*$-algebras, that is, the problem of determining which $C^*$-algebras can arise 
as $C^*$-algebras of various classes of Lie groups, in which special topological properties 
of unitary dual spaces of the groups under consideration hold a key role.

For noncommutative noncompact groups, their unitary dual spaces in general fail to have the Hausdorff property, 
and even the singleton subsets of the dual may not be closed subsets. 
Nevertheless, for instance in the case of nilpotent Lie groups, the unitary dual space 
is always a connected topological space having the property $T_1$, that is, its singleton subsets are closed. 
This $T_1$ property comes from the fact that for any nilpotent Lie group its $C^*$-algebra is liminal, that is,  
the image of any non-trivial irreducible $*$-representation is equal to the set of all compact operators on the representation space, 
and then every primitive ideal of the group $C^*$-algebra is a  maximal ideal. 
Using results on the structure of the space of coadjoint orbits, further specific topological features of the unitary duals of nilpotent Lie groups can be established, as we will show below. 

Before proceeding to a more detailed description of the present paper, 
it is worth mentioning that, while we present only properties of nilpotent Lie groups, the methods surveyed here are applicable to wider classes of Lie groups and have lead to interesting results in their representation theory. 
See for instance \cite{BB16a}, \cite{BB17}, and \cite{BB17b}. 

The contents of this paper are as follows. 
In Section~\ref{section2} we first recall some basic topological properties of spaces of primitive ideals of $C^*$-algebras. 
This motivates the aforementioned inverse problem for primitive ideal spaces, and in this connection we discuss the solution to that problem in the special case of nuclear $C^*$-algebras and we draw some direct consequences of the corresponding result. 
The role of this section is to provide the general framework for the study of topological aspects of group $C^*$-algebras in the next sections. 

In Section~\ref{section3} we discuss the notion of special $\RR$-space, which was introduced in \cite{BBL17} as a convenient tool in order to describe the unitary duals of nilpotent Lie groups. 
Loosely speaking, a special $\RR$-space is a topological space endowed with a continuous action of the multiplicative semigroup $\RR$. 
The space of coadjoint orbits of any Lie group $G$ has the  canonical structure of a special $\RR$-space, coming from the vector space structure of $\gg^*$. 
Therefore, if $G$ is nilpotent, connected and simply connected, then Kirillov's correspondence leads to the  canonical structure of a special $\RR$-space on the unitary dual~$\widehat{G}$. 
In this section we also discuss the notion of solvable topological space, again motivated by representation theory.  
More precisely, one of the main results of \cite{BBL17} is that the $C^*$-algebra of every connected simply connected nilpotent Lie group is a solvable $C^*$-algebra, and the spectrum of such a $C^*$-algebra is a solvable topological space. 
The background of that result is actually developed in 
Section~\ref{section-nilpotent} in some detail. 

In Section~\ref{section5} we provide an exact description of the quasi-compact subsets of the unitary dual of a Heisenberg group. 
This description is a version of the well-known fact that the compact subsets of a finite-dimensional vector space are the closed bounded subsets. 
Loosely speaking, we establish a result of this type in which the vector spaces are replaced by some topological spaces for which the Hausdorff property fails to hold, and this naturally leads to new phenomena. 
For instance, we find that the intersection of two quasi-compact sets may not be quasi-compact. 

Section~\ref{section6} includes some aspects of noncommutative dimension theory for the $C^*$-algebras of nilpotent Lie groups. 
Here we recall our earlier result from \cite{BB16b} to the effect that the dimension of the space of characters of an exponential Lie group is equal to the real rank of the $C^*$-algebra of that group. 
We then establish lower and upper estimates of the nuclear dimension of the $C^*$-algebra of a nilpotent Lie group. 

Finally, in Section~\ref{section7} we present further specific examples of nilpotent Lie groups  and 
in Appendix~\ref{AppA} we collect some basic topological terminology that we use in the main body of this paper.

\subsection*{Notation} 
Throughout this paper we denote the Lie groups by upper case Roman letters and their corresponding Lie algebras by the corresponding lower case Gothic letters. 
We denote by $\RR$ and $\CC$ the fields of real and complex numbers, respectively. 
We also denote $\RR^\times:=\RR\setminus\{0\}$ and $\TT:=\{z\in\CC\mid \vert z\vert=1\}$, and both these sets are usually regarded as 1-dimensional Lie groups with respect to the group operation given by multiplication. 
For notions and results on $C^*$-algebras we refer to \cite{Dix64} and \cite{BO08}. 

\section{On the inverse problem for primitive ideal spaces of $C^*$-algebras}\label{section2}
Let $\Ac$ be a $C^*$-algebra. 
We denote by $\widehat{\Ac}$ the set of unitary equivalence classes $[\pi]$ of non-zero irreducible $*$-repres\-en\-tations~$\pi$ of $\Ac$, 
and by~$\Prim(\Ac)$ its space of primitive ideals. 
We endow $\Prim(\Ac)$ with the Jacobson topology, for which a set $T\subseteq\Prim(\Ac)$ is closed if and only if $T=\{\Jc\in\Prim(\Ac)\mid I(T)\subseteq\Jc\}$, 
where $I(T):=\bigcap\limits_{\Ic\in T}\Ic$. 
By definition, there is the canonical surjective map 
\begin{equation}
\label{basic_map}
\widehat{\Ac}\to\Prim(\Ac),\quad [\pi]\mapsto\Ker\pi
\end{equation}
and we endow $\widehat{\Ac}$ with the weakest topology for which the above map is continuous, and then that map is open. 
Here is a list of classical properties of these objects: 

\begin{proposition}\label{basic}
If $\Ac$ is a $C^*$-algebra, then the following assertions hold: 
\begin{enumerate}[(i)]
	\item\label{basic_T0} 
	$\Prim(\Ac)$ has the property $T_0$. 
	\item\label{basic_typeI} 
	$\widehat{\Ac}$ has the property $T_0$ if and only if the map \eqref{basic_map} is injective, and this is the case if and only if $\Ac$ is a $C^*$-algebra of type~I. 
	\item\label{aH} 
	If $\Ac$ is a $C^*$-algebra of type~I, then $\widehat{\Ac}$ is almost Hausdorff. 
	\item\label{basic_Baire} 
	$\widehat{\Ac}$ is locally quasi-compact and is a Baire space.
	\item\label{basic_2nd-countable} 
	If $\Ac$ is separable, then $\widehat{\Ac}$ is second countable and spectral.
	\item\label{basic_liminal} 
	If $\Ac$ is liminal, then $\Prim(\Ac)$ has the property~$T_1$, and the converse of this assertion holds if $\Ac$ is separable and of type I.
\end{enumerate}
\end{proposition}

\begin{proof}
Assertion~\eqref{basic_T0} was noted in \cite[3.1.3]{Dix64}, 
while Assertion~\eqref{basic_typeI} was proved in \cite[3.1.6]{Dix64}. 
See \cite[\S 2.2(iv)]{BrPe09} for Assertion~\eqref{aH}. 
For Assertion~\eqref{basic_2nd-countable} see \cite[3.3.4]{Dix64} and \cite[Lemma~2.2]{Ki06}. 
Assertion~\eqref{basic_Baire} is contained in \cite[3.3.8 and 3.4.13]{Dix64}. 
For Assertion~\eqref{basic_liminal} see  \cite[4.7.15 and \S 9.1]{Dix64}. 
\end{proof}

We point out an important consequence of the above proposition, 
namely, that if $\Ac$ is a $C^*$-algebra of type~I, then the map \eqref{basic_map} is a homeomorphism.

A more recent result in  topology of the space of primitive ideals concerns nuclear $C^*$-algebras. 
Before stating it, we recall that a $C^*$-algebra $\Ac$ is \emph{nuclear} if there is a net of  finite dimensional $C^*$-algebras $F_\iota$  and   contractive  completely positive maps
$\varphi_\iota\colon \Ac \to F_\iota$ and $\phi_\iota \colon F_\iota\to \Ac$ such that 
$$ \Vert(\psi_\iota \circ \varphi_\iota)(a) - a\Vert \to 0$$
for every $a\in \Ac$. 
All type  I $C^*$-algebras are nuclear. (See \cite[Prop.~2.7.4]{BO08}.)

\begin{theorem}[{\cite[\S 5]{Ki06}}]\label{spectral}
Let $X$ be a topological space that is second countable, spectral, and has the property~$T_0$. 
Then the following assertions are equivalent: 
\begin{enumerate} [(i)]
	\item $X$ is homeomorphic to $\Prim(\Ac)$ for some separable nuclear $C^*$-algebra~$\Ac$. 
	\item There exist a locally compact Polish space $Y$ and a continuous map $\varphi\colon Y\to X$ with the following properties: 
	\begin{enumerate}
		\item For every locally closed subset $A\subseteq X$ one has $\varphi(Y)\cap A\ne\emptyset$. 
		\item One has $\overline{\bigcup\limits_{n\ge 1}\varphi^{-1}(F_n)}=\varphi^{-1}\Bigl(
		\overline{\bigcup\limits_{n\ge 1}F_n}\Bigr)$ 
		for every increasing sequence $\{F_n\}_{n\ge 1}$ in $\Fc(X)$.
	\end{enumerate}
\end{enumerate}
\end{theorem}

\begin{lemma}[{\cite[Lemma 2.3]{Ki06}}]\label{cont-open}
If $X$ and $Y$ are topological spaces, then an onto map $\varphi\colon Y\to X$ is continuous and open if and only if for every subset $Z\subseteq X$ one has $\overline{\varphi^{-1}(Z)}=\varphi^{-1}(\overline{Z})$. 
\end{lemma}

Now we can draw the following simple consequences of Theorem~\ref{spectral}.

\begin{corollary}\label{cont-open-cor1}
Let $X$ be a topological space that is 
spectral and has the property~$T_0$. 
If there exist a locally compact Polish space $Y$ and a surjective map $\varphi\colon Y\to X$ which is continuous and open, then $X$ is homeomorphic to $\Prim(\Ac)$ for some separable nuclear $C^*$-algebra~$\Ac$.
\end{corollary}

\begin{proof}
	Since $Y$ is second countable and the map $\varphi\colon Y\to X$ is continuous, open, and surjective, it is straightforward to prove that $X$ is second countable as well. 
Use then Theorem~\ref{spectral} along with Lemma~\ref{cont-open}. 
\end{proof}

\begin{corollary}\label{cont-open-cor2}
	If $X$ is a topological space that has the property~$T_1$  
	and there exist a locally compact Polish space $Y$ and a surjective map $\varphi\colon Y\to X$ which is continuous and open, then $X$ is homeomorphic to $\Prim(\Ac)$ for some separable nuclear $C^*$-algebra~$\Ac$, and every type I  $C^*$-algebra with this property is liminal.
\end{corollary}

\begin{proof}
	As $X$ has the property~$T_1$, it also has the property $T_0$ and is a spectral space. 
	It now follows by Corollary~\ref{cont-open-cor1} that $X$ is homeomorphic to $\Prim(\Ac)$ for some separable nuclear $C^*$-algebra~$\Ac$. 
	Finally, by Proposition~\ref{basic} \eqref{basic_liminal}, since $\Ac$ is of type I,  separable and  $\Prim(\Ac)$ is $T_1$, it follows that $\Ac$ is liminal. 
\end{proof}

\begin{corollary}\label{cont-open-cor3}
	Let $Y$ be any locally compact Polish space and $G$ be any group with a group action $G\times Y\to Y$ whose orbits are locally closed subsets of $Y$. 
	Then the corresponding orbit space $Y/G$ endowed with its quotient topology is homeomorphic to $\Prim(\Ac)$ for some 
	separable nuclear $C^*$-algebra $\Ac$. 
	All the orbits of the group action of $G$ in $Y$ are closed if and only if any/every type I $C^*$-algebra with the above property is liminal. 
\end{corollary}

\begin{proof} 
	It is well known that for every group action on a topological space its corresponding quotient map on the orbit space is continuous, open, and surjective. 
	Moreover, the orbit space has the property $T_0$, respectively $T_1$, if and only if all orbits are locally closed, respectively all orbits are closed (see for instance the proof of \cite[Prop. 2.5]{BB17}). 
	Thus the assertion follows by Corollaries \ref{cont-open-cor1}--\ref{cont-open-cor2}. 
\end{proof}

\begin{remark}
	\normalfont
Using the same method of proof, the assertion of Corollary~\ref{cont-open-cor3} carries over from transformation groups to more general groupoids. 
More specifically, let $\Gc\rightrightarrows Y$ be a locally compact groupoid having a Haar system, whose orbits are locally closed, and whose base $Y$ is a Polish space. 
Then the corresponding orbit space $\Gc\setminus Y$ endowed with its quotient topology is homeomorphic to $\Prim(\Ac)$ for some 
separable nuclear $C^*$-algebra $\Ac$. 
All the orbits of the groupoid $\Gc$ are closed subsets of $Y$ if and only if any/every type~I  $C^*$-algebra with the above property is liminal. 
\end{remark}

We now recall from \cite{Ki06} the notion of generalized Gelfand transform. 

\begin{definition}
\normalfont 
Let $\Ac$ be a $C^*$-algebra. 
For any $a\in\Ac$, its \emph{generalized Gelfand transform} is the function 
$$N(a)\colon\Prim(\Ac)\to[0,\infty),\quad N(a)(\Jc):=\Vert a+\Jc\Vert$$
where we used the $C^*$-norm of the primitive quotient $\Ac/\Jc$ for any $\Jc\in\Prim(\Ac)$. 
\end{definition}

\begin{theorem}[{\cite[Th. 1.5]{Ki06}}]\label{Dini1}
For every $C^*$-algebra $\Ac$ the set of all Dini functions on $\Prim(\Ac)$ is exactly $\{N(a)\mid a\in\Ac\}$. 
If $\Ac$ is separable, then $\Prim(\Ac)$ is a Dini space. 
\end{theorem}

\section{Solvable topological spaces and special $\RR$-spaces}\label{section3}

In this section we provide the topological background 
that is needed for describing some special properties of unitary dual spaces of connected simply connected nilpotent Lie groups in Section~\ref{section-nilpotent}.  
The notions of solvable topological spaces and special $\RR$-spaces that were introduced in \cite{BBL17}, 
and we discuss them here in some detail along with many examples.

\begin{definition}\label{solvable_space}
	\normalfont
	A topological space $X$ is called \emph{solvable} if it is second countable, is locally quasi-compact, has the property $T_0$, and 
	there exists an increasing finite family of open subsets, 
	called a \emph{solving series},  
	$$\emptyset=V_0\subseteq V_1\subseteq\cdots\subseteq V_n=X, $$
	for which the set $V_j\setminus V_{j-1}$ is Hausdorff in its relative topology for $j=1,\dots,n$.
\end{definition}

\begin{remark}
	\normalfont
	In Definition~\ref{solvable_space} the set 
	$V_j\setminus V_{j-1}$ is actually locally compact in its relative topology for $j=1,\dots,n$. 
\end{remark}

\begin{definition}\label{strongly_solvable_space}
	\normalfont
	A topological space $X$ is called \emph{strongly solvable} if it is solvable, has the property $T_1$, and 
	has a solving series,  
	$\emptyset=V_0\subseteq V_1\subseteq\cdots\subseteq V_n=X$ 
	for which the set $V_j\setminus V_{j-1}$ is dense in  $X\setminus V_{j-1}$ for $j=1,\dots,n$.
\end{definition}

\begin{definition}\label{R-space}
	\normalfont
	A \emph{special $\RR$-space} is a topological space $X$ endowed with a 
	continuous map $\RR\times X\to X$, $(t,x)\mapsto t\cdot x$, and with a distinguished point $x_0\in X$ 
	satisfying the following conditions: 
	\begin{enumerate}
		\item For every $x\in X$ and $t\in \RR$ one has $0\cdot x= t \cdot x_0=x_0$ and $1\cdot x=x$. 
		\item For all $t,s\in\RR$ and $x\in X$ one has $t\cdot(s\cdot x)=ts\cdot x$. 
		\item\label{R-space_item3}
		 For every $x\in X\setminus \{x_0\}$ the map $\psi_x\colon \RR\to X$, $t\mapsto t\cdot x$ is a homeomorphism onto its image. 
	\end{enumerate}
	An \emph{$\RR$-subspace} of the  special $\RR$-space $X$ is any subset $\Gamma\subseteq X$ 
	such that $\RR^\times \cdot \Gamma\subseteq\Gamma$. 
	If this is the case, then $\Gamma\cup\{x_0\}$ is a special $\RR$-space on its own. 
	
	If $Y$ is another special $\RR$-space with its structural map $\RR\times X\to X$, $(t,x)\mapsto t\cdot x$ and its distinguished point $y_0\in Y$, then a map $\psi\colon X\to Y$ is called an \emph{isomorphism of special $\RR$-spaces} if $\psi$ is a homeomorphism and $\psi(t\cdot x)=t\cdot \psi(x)$ for all $t\in\RR$ and $x\in X$. 
	
	In the above framework, a function $\varphi\colon X\to\RR$ is called \emph{homogeneous} if there exists $r\in[0,\infty)$ 
	such that $\varphi(t\cdot x)=t^r\varphi(x)$ for all $t\in\RR$ and $x\in X$. 
\end{definition} 

\begin{remark}\label{R-space_rem}
	\normalfont
If $X$ is a special $\RR$-space, then one has: 
\begin{enumerate}
	\item If $t\in\RR^\times$ and $x\in X\setminus\{x_0\}$, then $t\cdot x=x_0$ if and only if $t=0$. 
	
	In fact, since $0\cdot x=x_0$, the equality  $t\cdot x=x_0$ is equivalent to $\psi_x(t)=\psi_x(0)$. 
	Using the fact that the map $\psi_x$ is injective, the latter equality is equivalent to $t=0$.   
	\item If $\psi\colon X\to Y$ is  an isomorphism of special $\RR$-spaces then $\psi(x_0)=y_0$. 
	
	In fact, $\psi(x_0)=\psi(0\cdot x_0)=0\cdot\psi(x_0)=y_0$. 
	\item The space $X$ is pathwise connected. 
\end{enumerate}
\end{remark}

\begin{example}[semi-algebraic cones]
	\label{semialg}
	\normalfont
	Every finite-dimensional real vector space is a special $\RR$-space. 
	Moreover, if $\varphi_1,\dots,\varphi_{n_1},\psi_1,\dots,\psi_{n_2}\colon\RR^m\to\RR$ 
	are any homogeneous polynomials, then the semi-algebraic cone 
	$$\Gamma:=\{x\in\RR^n\mid\varphi_{j_1}(x)=0\ne\psi_{j_2}(x)\text{ for }1\le j_1\le n_1\text{ and }1\le j_2\le n_2\}$$
	is an $\RR$-subspace of $\RR^m$ in the sense of Definition~\ref{R-space}. 
\end{example}

\begin{example}[orbit spaces of linear group actions]
	\label{orbits}
	\normalfont
	The orbit space of any linear group action has the natural structure of a special $\RR$-space. 
	We explain this in the case of the coadjoint action, for later use. 
	
	If  $G$ is a connected and simply connected  nilpotent Lie group, then the special $\RR$-space structure of $\gg^*$ 
	gives rise to a special $\RR$-space structure of the orbit space $\gg^*/G$ (hence also of the unitary dual space $\widehat{G}$ 
	via Kirillov's correspondence), and the vector space of characters $[\gg,\gg]^\perp$, 
	viewed as the set of singleton orbits, is an $\RR$-subspace of $\gg^*/G$. 
	More specifically, the special $\RR$-space structure of $\gg^*/G$ is given by the map 
	$$\RR\times (\gg^*/G)\to\gg^*/G, \quad (t,\Oc_{\xi})\mapsto\Oc_{t\xi}$$
	where we denote by $\Oc_\xi$ the coadjoint orbit of every $\xi\in\gg^*$. 
\end{example}

\begin{remark}
	\normalfont
Assume the setting of Example~\ref{orbits}. 
Using the canonical homeomorphism  
$\widehat{G}\simeq\gg^*/G$, 
it follows that $\widehat{G}$ carries a continuous action of the group $\RR^\times$. 
Pulling back this group action by the group homomorphism $\RR\to\RR^\times$, $t\mapsto \ee^t$, 
we then obtain a natural action of $(\RR,+)$ on $\widehat{G}$, 
$\alpha\colon\RR\times\widehat{G}\to\widehat{G}$.  

On the other hand, if $\beta\colon\RR\times G\to G$ is any continuous action of $(\RR,+)$ by automorphisms of $G$, 
this defines a continuous action $\widehat{\beta}\colon\RR\times\widehat{G}\to\widehat{G}$, 
$(t,[\pi])\mapsto[\pi\circ\beta_t]=:\widehat{\beta}_t$.  
Therefore one may wonder whether $\alpha=\widehat{\beta}$ for a suitable $\beta$. 
In other words, whether the above canonical flow $\alpha$ on $\widehat{G}$ comes from some 1-parameter automorphism group of~$G$. 
This is clearly the case if $G=(\Vc,+)$ for some finite-dimensional real vector space, however it does not hold true for any nilpotent Lie group~$G$. 

For instance, let $G$ be a connected simply connected nilpotent Lie group whose Lie algebra $\gg$ is characteristically nilpotent, that is, every derivation of $\gg$ is a nilpotent linear map. 
(See for instance \cite[Ex.~7.5]{BBG15} for a self-contained discussion of a specific example of characteristically nilpotent Lie algebra.) 
Let us assume $\alpha_t=\widehat{\beta}_t$ for every $t\in\RR$, for a suitable $\beta$ as above. 
Since $\beta$ is a 1-parameter automorphism group of $G$, there exists a derivation $D\colon\gg\to\gg$ with $\beta_t=\ee^{tD}$ for every $t\in\RR$. 
As $D$ is a derivation of $\gg$, one has $D([\gg,\gg])\subseteq[\gg,\gg]$, hence there is a well-defined map $\widetilde{D}\colon \gg/[\gg,\gg]\to\gg/[\gg,\gg]$, $x+[\gg,\gg]\mapsto Dx+[\gg,\gg]$. 
As $D$ is nilpotent, so are both $\widetilde{D}$ and its dual map $\widetilde{D}^*\colon (\gg/[\gg,\gg])^*\to(\gg/[\gg,\gg])^*$. 
We now note that there is a canonical linear isomorphism 
from $(\gg/[\gg,\gg])^*$ onto the space of characters $[\gg,\gg]^\perp$ and via this isomorphism the map $\ee^{t\widetilde{D}^*}$ is the restriction of $\widehat{\beta}_t$ to $[\gg,\gg]^\perp$ ($\hookrightarrow\widehat{G}$) since $\ee^{tD}=\beta_t$. 
If we assume $\widehat{\beta}=\alpha$, then 
$\ee^{t\widetilde{D}^*}\xi=\widehat{\beta}_t(\xi)=\alpha_t(\xi)=\ee^t\xi$ for every $\xi\in(\gg/[\gg,\gg])^*\simeq[\gg,\gg]^\perp\hookrightarrow\widehat{G}$ and for every $t\in\RR$, hence $\widetilde{D}^*\xi=\xi$, which is impossible since we have seen above that the linear map $\widetilde{D}^*$ is nilpotent. 
\end{remark}

\begin{example}[complement of an $\RR$-subspace]\label{complem}
	\normalfont
If $X$ is a special $\RR$-space and $\Gamma\subseteq X$ is an $\RR$-subspace of $X$, then $X\setminus\Gamma$ is also an $\RR$-subspace of~$X$.

In fact, we must check that if $x\in X\setminus\Gamma$ and $t\in\RR^\times$, then $t\cdot x\in X\setminus\Gamma$. 
Assuming $t\cdot x\in\Gamma$, we obtain $t^{-1}\cdot(t\cdot x)\in\Gamma$, that is, $x\in\Gamma$, which is a contradiction. 
\end{example}

\begin{example}[the special $\RR$-space $\RR$]\label{1dim}
	\normalfont
If $X$ is a special $\RR$-space and $X$ is homeomorphic to $\RR$, then for every $x\in X\setminus\{x_0\}$ the map 
$\psi_x\colon\RR\to X$, $t\mapsto t\cdot x$, is a homeomorphism, and moreover $\psi_x$ is an isomorphism of special $\RR$-spaces between $X$ and $\RR$ (regarded as a 1-dimensional real vector space). 

In fact, if $x\in X\setminus\{x_0\}$, 
then the map $\psi_x$ is a homeomorphism onto its image by Definition~\ref{R-space}, hence it remains to prove that $\psi_x$ is surjective. 
Since $\psi_x$ is continuous and $\RR$ is connected, it follows that $\psi_x(\RR)$ is a connected subset of $X$. 
By hypothesis there exists a homeomorphism 
$\theta\colon X\to\RR$. 
Since the connected subsets of $\RR$ are the intervals, 
there exists an interval $J_x\subseteq\RR$ for which $\theta\circ \psi_x\colon \RR\to J_x$ is a homeomorphism. 
If the interval $J$ is not open, then there exists $a_0\in J$ 
(namely one endpoint of $J$ that belongs to $J$) such that $J_x\setminus\{a_0\}$ is connected. 
Since  $\theta\circ \psi_x$ is a homeomorphism, it then follows that there exists the point $t_0:=(\theta\circ \psi_x)^{-1}(a_0)\in\RR$ for which $\RR\setminus\{t_0\}$ is connected, which is impossible. 
Therefore $J_x$ is an open interval. 
Moreover, the point $b_0:=\theta(x_0)\in\RR$ is independent of $x\in X$ and one has $b_0=\theta(\psi_x(0)\in J_{x_0}$ for every $x\in X\setminus\{x_0\}$. 

Now let $x_1,x_2\in X\setminus\{x_0\}$ arbitrary. 
Since $J_{x_1}$ and $J_{x_2}$ are open intervals in $\RR$ whose intersection contains the point $b_0$, it follows that  $(J_{x_1}\cap J_{x_2})\setminus\{b_0\}\ne\emptyset$. 
Then there exist $t_1,t_2\in\RR\setminus\{0\}$ 
with $\theta(\psi_{x_1}(t_1))=\theta(\psi_{x_2}(t_2))$, 
hence $\psi_{x_1}(t_1)=\psi_{x_2}(t_2)$, 
that is, $t_1\cdot x_1=t_2\cdot x_2$.  
On the other hand, it follows at once by Definition~\ref{R-space} that $x=1\cdot x\in\psi_x(1)$ for every $x\in X$. 
Hence, since $t_1\in\RR\setminus\{0\}$, 
we obtain $x_1=(t_1^{-1}t_2)\cdot x_2\in\psi_{x_2}(\RR)$ 
for all $x_1,x_2\in X\setminus\{x_0\}$. 
This shows that $\psi_x\colon\RR\to X$ is indeed surjective for every $x\in X$. 
\end{example}

\begin{definition}\label{solvspecl_space}
	\normalfont
	A topological space $X$ is a \emph{special solvable space} if 
	it is strongly solvable and  
	it has a solving series as in Definition~\ref{strongly_solvable_space} with the following additional properties: 
	\begin{enumerate}
		\item\label{solvspecl_space_item1} 
		$X$ has the structure of a special $\RR$-space 
		and $\Gamma_j:=V_j\setminus V_{j-1}\subseteq X$ is an $\RR$-subspace for $j=1,\dots,n$. 
		\item\label{solvspecl_space_item2} 
		$\Gamma_n$ is isomorphic as a special $\RR$-space to a finite-dim\-ensional vector space, whose origin is the distinguished point of $X$. 
		\item\label{solvspecl_space_item3} 
		For $j=1,\dots,n-1$ the points of $\Gamma_{j+1}$ are closed and separated in $X\setminus V_{j}$.  
		\item For $j=1,\dots,n$, $\Gamma_j$ is isomorphic as a special $\RR$-space to a semi-algebraic cone $C_j$
		in a finite-dimensional vector space. 
		In addition, $C_1$ is assumed to be a Zariski open set,  
		and the dimension of the corresponding ambient vector space is called the \emph{index of $X$} and is denoted by $\ind X$. 
		\item For $j=1,\dots,n$, there exists a homogeneous function $\varphi_j\colon X\to\RR$  
		such that $\varphi_j\vert_{\Gamma_1}$ is a polynomial function (via the isomorphism $\Gamma_1\simeq C_1$) 
		and 
		$$\Gamma_j=\{\gamma\in X\mid \varphi_j(\gamma)\ne0\text{ and }\varphi_i(\gamma)=0\text{ if }i<j\}.$$ 
	\end{enumerate}
We then say that a solving series with the above properties is a \emph{special solving series} of~$X$. 
The least integer $n\ge 1$ for which there exists 
a special solving series   
as above is called the \emph{length} of $X$ and is denoted $\lgth X$. 
\end{definition}

\begin{lemma}\label{dom}
	Let $X$ be any topological space and for $j=1,2$ let $V_j$ be any open subset of $X$ 
	that is homeomorphic to an open subset of $\RR^{r_j}$, where $r_j\ge1$ is some integer. 
	If $V_1\cap V_2\ne\emptyset$, then $r_1=r_2$. 
\end{lemma}

\begin{proof}
See \cite[Lemma 2.10]{BBL17}. 
\end{proof}

\begin{remark}\label{ind_invar}
	\normalfont
Assume the setting of Definition~\ref{solvspecl_space}. 
Since $\Gamma_1$ is open and dense in $X$, 
it follows by Lemma~\ref{dom} that $\ind X$ does not depend on the choice of the solving series of~$X$. 
\end{remark}

\begin{example}[length 1]\label{length1}
	\normalfont
	The special solvable spaces of length 1 are exactly the finite-dimensional real vector spaces, up to an isomorphism of special $\RR$-spaces. 

Indeed, it is clear that for every finite-dimensional real vector space $\Vc$ the solving series $\emptyset =V_0\subseteq V_1=\Vc$ has the properties from Definition~\ref{solvspecl_space}, with $\Gamma_1=C_1=\Vc$, 
and the role of $\varphi_1\colon\Vc\to\RR$ can be played by any nonzero constant function. 

Conversely, let $X$ be a special solvable space  
with $\lgth X=1$.  
Then $X$ is a special $\RR$-space for which there exists a solving series $\emptyset =V_0\subseteq V_1=X$ 
satisfying the conditions from Definition~\ref{solvspecl_space} for $n=1$. 
Then $X=\Gamma_1$,  and  
$\Gamma_1$ is isomorphic as a special $\RR$-space to a finite-dim\-ensional vector space by Definition~\ref{solvspecl_space}\eqref{solvspecl_space_item2}.
\end{example}

\begin{example}[length 2, index 1]\label{length2index1}
	\normalfont
	Let $X$ be a special solvable space  
	such that $\lgth X=2$ and $\ind X=1$.  
	Then $X$ is a special $\RR$-space for which there exists a solving series $\emptyset =V_0\subseteq V_1\subseteq V_2=X$ 
	satisfying the conditions from Definition~\ref{solvspecl_space} for $n=2$. 
	
	The set 
	$\Gamma_2:=X\setminus V_1$ is isomorphic as a special $\RR$-space to a finite-dim\-ensional vector space by Definition~\ref{solvspecl_space}\eqref{solvspecl_space_item2}. 
	On the other hand, the set $V_1:=\Gamma_1$ is isomorphic as a special $\RR$-space to a an open semi-algebraic cone $C_1$
	in $\RR$, since $\ind X=1$. 
	It is easily checked the only open cones in $\RR$ that are $\RR$-subspaces of $\RR$ are  
	$\RR$ and $\RR^\times$.  
	But one cannot have $C_1=\RR$ because for $x=0\in C_1$ one obtains a contradiction with Definition~\ref{R-space}\eqref{R-space_item3}. 
	
	Consequently any special solvable space $X$ 
	with $\lgth X=2$ and $\ind X=1$ is of the form $X=\Gamma_1\sqcup\Gamma_2$, where $\Gamma_2$ is a finite-dimensional real vector space $\Vc$ and $\Gamma_1=\RR^\times$. 
	Taking into account that $\Gamma_1$ is dense in $X$ (see Definition~\ref{strongly_solvable_space}), 
the topology of $X$ is such that $\Gamma_2$ is a closed subset of $X$ and is equal to the boundary of $\Gamma_1$ in~$X$. 
\end{example}

\section{Nilpotent Lie groups and the topology of their unitary duals}\label{section-nilpotent}
In this section, unless otherwise mentioned, we denote by $G$ an arbitrary connected simply connected nilpotent Lie group  with its corresponding Lie algebra $\gg$ having 
a fixed Jordan-H\"older sequence 
$$\{0\}=\gg_0\subseteq \gg_1\subseteq \cdots\subseteq\gg_{m-1} \subseteq\gg_m=\gg,$$
and we select $X_j\in\gg_j\setminus\gg_{j-1}$ for $j=1,\dots,m$. 
We denote by 
$$\langle\cdot,\cdot\rangle\colon\gg^*\times\gg\to\RR$$
the duality pairing between $\gg$ and its linear dual space~$\gg^*$, 
and for every subalgebra $\hg\subseteq\gg$ we define 
$$(\forall\xi\in\gg^*)\quad \hg(\xi):=\{X\in\hg\mid(\forall Y\in\hg)\  \langle\xi,[X,Y]\rangle=0\}, $$
and $ \hg^\perp= \{\xi \in \gg^*\mid \hg \subset \ker \xi\}$. 
We also denote by $\Ad_G^*$ the coadjoint action of $G$ in $\gg^*$.
There is a canonical map $q\colon \gg^* \to \gg^*/G$, $\xi \mapsto \Oc_\xi$, where 
$\Oc_\xi = \Ad^*_G(G)\xi$. 
The space $\gg^*/G$ of coadjoint orbits is endowed with its quotient topology.

Let $\Ec$ be the set of all subsets of $\{1,\dots,m\}$ 
endowed with the total ordering defined for all $e_1\ne e_2$ in $\Ec$ by 
$$e_1\prec e_2\iff\min(e_1\setminus e_2)<\min(e_2\setminus e_1)$$
where we use the convention $\min\emptyset=\infty$, so in particular $\max\Ec=\emptyset$.

We define the jump indices
$$(\forall \xi\in\gg^*)\quad J_\xi:=\{j\in\{1,\dots,m\}\mid \gg_j\not\subset\gg(\xi)+\gg_{j-1}\}$$
and 
$$(\forall e\in\Ec)\quad \Omega_e:=\{\xi\in\gg^*\mid J_\xi=e\}.$$
The \emph{coarse stratification of $\gg^*$} is the family $\{\Omega_e\}_{e\in\Ec}$, 
which is a finite partition of~$\gg^*$ consisting of $G$-invariant sets. 
For every coadjoint $G$-orbit $\Oc\in\gg^*/G$ we define $J_{\Oc}:=J_\xi$ for any $\xi\in\Oc$ and then 
\begin{equation}\label{coarse_eq1}
(\forall e\in\Ec)\quad \Xi_e:=\{\Oc\in\gg^*/G\mid J_{\Oc}=e\}.
\end{equation}

\begin{lemma}\label{coarse_lemma}
	Assume the above setting and for every $\Oc\in\gg^*/G$ pick any unitary irreducible representation 
	$\pi_{\Oc}\colon G\to \Bc(\Hc_{\Oc})$ which is associated with~$\Oc$ via Kirillov's correspondence. 
	If we endow the space $\gg^*/G\simeq\widehat{G}$ with its canonical topology, 
	then for every index set $e\in\Ec$ the following assertions hold: 
	\begin{enumerate}
		\item\label{coarse_lemma_item1} The relative topology of $\Xi_e\subseteq\gg^*/G$ is Hausdorff. 
		\item\label{coarse_lemma_item2} For every  $\phi\in \Cc_c^\infty(G)$ 
		the function 
		$$\Xi_e\to\CC,\quad \Oc\mapsto\Tr(\pi_{\Oc}(\phi)) $$
		is well defined and continuous. 
	\end{enumerate}
\end{lemma}

\begin{proof}
	See \cite[Lemma 4.1]{BBL17}.
\end{proof}

\begin{definition}
	The cardinal of the set $\{e\in\Ec\mid \Omega_e\ne\emptyset\}$ (depending on the fixed Jordan-H\"older sequence of $\gg$) is called the \emph{coarse length} of the nilpotent Lie group $G$ and we denote it by $\clgth(G)$. 
\end{definition}

\begin{theorem}\label{dual_spec}
For any connected simply connected nilpotent Lie group $G$, its unitary dual $\widehat{G}$ is a special solvable space 
and $\lgth\widehat{G}\le \clgth(G)$. 
\end{theorem}

\begin{proof}
This follows by \cite[Th. 4.11]{BBL17}. 
\end{proof}

\begin{example}\label{filif}
	\normalfont
Let us assume that $m:=\dim\gg\ge 3$ and that the nilpotent Lie algebra $\gg$ has a basis $X_1,\dots,X_m$ with the commutation relations $$[X_m,X_j]=X_{j-1} \quad \text{for} \quad j=1,\dots,m-1, $$
where $X_0:=0$, 
and $[X_k,X_j]=0$ if $1\le j\le k\le m-1$.  
Defining $\gg_j:=\spa\{X_1,\dots,X_j\}$ for $j=1,\dots,m$, 
we obtain a Jordan-H\"older sequence in $\gg$, and 
we will show that $\clgth(G)=m-1$ for this Jordan-H\"older sequence.

To this end we first note that $[\gg,\gg]=\spa\{X_j\mid 1\le j\le m-2\}$ hence 
\begin{equation}
\label{filif_eq1}
[\gg,\gg]^\perp=\{\xi\in\gg^*\mid \langle\xi,X_j\rangle=0\text{ for }j=1,\dots,m-2\}.
\end{equation}
Now let $\xi\in\gg^*\setminus[\gg,\gg]^\perp$ and define  $j_1\in\{2,\dots,m\}$ by the conditions $$\langle\xi,X_{j_1-1}\rangle\ne0=\langle\xi,X_j\rangle\quad \text{for  every } \; j<j_1-1.$$
We check that $J_\xi=\{j_1,m\}$. 

One has $X_j\in\gg(\xi)$ for $j\in\{1,\dots,j_1-1\}$ since  $[X_m,X_j]=X_{j-1}\in\Ker\xi$ 
and  $[X_k,X_j]=0\in\Ker\xi$ for all $k\in\{1,\dots,m-1\}$. 
Therefore $\gg_{j_1-1}\subseteq\gg(\xi)$. 

Since $\xi\not\in[\gg,\gg]^\perp$, one has $j_1\le m-1$ by \eqref{filif_eq1}. 
Then $[X_m,X_{j_1}]=X_{j_1-1}$, hence $\langle\xi,[X_m,X_{j_1}]\rangle\ne 0$. 
This implies $\gg(\xi)\cap\spa\{X_m,X_{j_1}\}=\{0\}$, 
and then it is straightforward to check that $j_1,m\in J_\xi$. 

On the other hand, it is well known that  $\vert J_\xi\vert =\dim\Oc_\xi=2$. (See for instance \cite[Prop.~11]{AM10} for a direct proof of this fact.)
Therefore $J_\xi=\{j_1, m\}$. 
This shows that  
$$\Ec=\{\{2, m\}, \dots, \{m-2, m\}, \emptyset\}, $$
hence $\clgth(G)=m-1$. 
\end{example}

The following lemma, interesting on its own, plays also a key role in the proof of Proposition~\ref{maximal}.

\begin{lemma}\label{quot}
Let $G$ be a connected simply connected nilpotent Lie group with its corresponding Lie algebra $\gg$. 
If $q\colon\gg^*\to\gg^*/G$, $\xi\mapsto\Oc_{\xi}$, 
is the quotient map onto the set of coadjoint orbits, 
then the following assertions hold: 
\begin{enumerate} 
\item 
The set of characters $[\gg,\gg]^\perp\subseteq\gg^*$, regarded as the set of all singleton coadjoint orbits, is a closed subset of $\widehat{G}$ whose complement is either empty (if $G$ is commutative) 
	or dense in $\widehat{G}$ (if $G$ is noncommutative). 
\item For any dense open set $D\subseteq\gg^*/G$, the set $q^{-1}(D)\subseteq\gg^*$ is also open and dense.  
\end{enumerate}
\end{lemma}

\begin{proof}
The quotient map $q\colon \gg^*\to\gg^*/G$ is a continuous open surjective map, hence it takes every open dense subset of $\gg^*$ onto an open dense subset of $\gg^*/G$. 
Now the first assertion follows since a straightforward reasoning of linear algebra shows that the set $\{\xi\in\gg^*\mid [\gg,\gg]\not\subset\Ker\xi\}$ either is empty or is a dense open subset of $\gg^*$, and the image of this set in $\gg^*/G$ is exactly the set of all coadjoint orbits that are not singleton sets.

The second assertion is a direct consequence of the fact that the map $q$ is open, continuous, and surjective. 
\end{proof}

\begin{proposition}\label{maximal}
	Let $G$ be a connected simply connected nilpotent Lie group with its corresponding Lie algebra $\gg$.  
	Regard the set of characters $[\gg,\gg]^\perp\subseteq\gg^*$ as a subset of  $\widehat{G}$. 
	If $S$ is a closed subset of $\widehat{G}$ 
	with $[\gg,\gg]^\perp\subseteq S\subseteq\widehat{G}$ and the relative topology of $S$ is Hausdorff, then $[\gg,\gg]^\perp$ is a closed-open subset of $S$. 
	
	In particular, if $S$ is moreover a connected subset of $\widehat{G}$ (e.g., an $\RR$-subspace), then $S=[\gg,\gg]^\perp$. Thus $[\gg,\gg]^\perp$ is a maximal connected closed subset of $\widehat{G}$ whose relative topology is Hausdorff. 
\end{proposition}

\begin{proof}
See \cite{BB17c}. 
\end{proof}

\begin{remark} 
	\normalfont
	To put Proposition~\ref{maximal} into perspective, 
	let us make the following elementary remarks. 
Let $X$ be a topological space. 
For any connected subset $S\subseteq X$ its closure $\overline{S}\subseteq X$ is also connected. 
This implies that any maximal connected subset of $X$ is closed, 
and the maximal connected subsets of $X$ are actually its connected components. 

However, if the relative topology of $S\subseteq X$ is Hausdorff, then the relative topology of its closure $\overline{S}\subseteq X$ need not be Hausorff. 
This shows that in general a topological space may not have any subset which is a maximal element of the set $\Sc_X$ of all closed connected Hausdorff subsets. 
We also note that this set $\Sc_X$ may not be inductively ordered with respect to the inclusion, hence Zorn's lemma cannot be used to obtain maximal elements of $\Sc_X$. 
\end{remark}

One has the following strengthening of \cite[Prop. 5.1]{BBL17}, which shows that the Heisenberg groups can be singled out from among the nilpotent Lie groups using only 
the $\RR$-space structures of their unitary dual spaces.  

\begin{proposition}\label{1H}
If $G$ is a connected simply connected nilpotent Lie group, then the following assertions are equivalent: 
\begin{enumerate}
	\item One has $\ind\widehat{G}=1$ and $\lgth\widehat{G}=2$.
	\item The Lie group $G$ is a Heisenberg group.  
\end{enumerate}
\end{proposition}

\begin{proof}
See \cite{BB17c}. 
\end{proof}

\section{Special topological properties of the unitary duals of Heisenberg groups}\label{section5}
In this section we establish some special properties of the quasi-compact subsets of the unitary duals of Heisenberg groups. 
We recall that if a $C^*$-algebra $\Ac$ is separable type~I, then the quasi-compact subsets of its spectrum $\widehat{\Ac}$ might not be Borel. (See \cite[4.7.10]{Dix64}.)

We need the following generalization of the well-known fact that the compact subsets of Hausdorff spaces are closed subsets. 

\begin{lemma}\label{qc1}
Let $X$ be a topological space having the property~$T_1$. 
If $X_0\subseteq X$ is a subset whose points are separated in $X$
(in particular, the relative topology of $X_0$ is Hausdorff), then for every quasi-compact subset $C\subseteq X$ the set $C\cap X_0$ is relatively closed in $X_0$. 
\end{lemma}

\begin{proof}
We  argue by contradiction. 
If $C\cap X_0$ is not relatively closed in $X_0$, 
then there exists a point $x_0\in X_0\setminus (C\cap X_0)=X_0\setminus C$ such that $x_0$ belongs to the relative closure of $C\cap X_0$ in $X_0$.

For every $c\in C$ one has $x_0\ne c$ hence, 
by the hypothesis that $X$ has the property~$T_1$ and the point $x_0\in X_0$ is separated in $X$, one obtains some 
open subsets $V_c,W_c\subseteq X$ with $V_c\cap W_c=\emptyset$, $x_0\in V_c$ and $c\in W_c$.

In particular $C\subseteq\bigcup\limits_{c\in C}W_c$ hence, 
since $C$ is quasi-compact, there exists a finite subset $C_0\subseteq C$ with $C\subseteq\bigcup\limits_{c\in C_0}W_c$. 
We then define $V:=\bigcap\limits_{c\in C_0}V_c$, 
which is the intersection of finitely many open subsets of $X$ 
containing $x_0$, hence $V$ is in turn an open subset of $X$ with $x_0\in V$. 
Moreover, 
$V\cap\bigcup\limits_{c\in C_0}W_c=\emptyset$ and $x_0\not\in\bigcup\limits_{c\in C_0}W_c$, hence  $V\cap C=\emptyset$. 
This implies $(V\cap X_0)\cap(C\cap X_0)=\emptyset$, 
which contradicts the assumption that $x_0$ belongs to the relative closure of $C\cap X_0$ in $X_0$.
\end{proof}

\begin{lemma}\label{qc2}
	Let $X$ be a topological space. 
	If $X_1\subseteq X$ is a closed subset, then for every quasi-compact subset $C\subseteq X$ the set $C\cap X_1$ is  quasi-compact with respect to the relative topology of $X_1$. 
\end{lemma}

\begin{proof}
Let $\{W_j\}_{j\in J}$ be an arbitrary open cover of $C\cap X_1$ in the relative topology of $X_1$. 
Hence $C\cap X_1\subseteq \bigcup\limits_{j\in J}W_j$ 
and for every $j\in J$ there exists an open subset $V_j$ of $X$ with $W_j=V_j\cap X_1$. 
Since $X_1$ is a closed subset of $X$, it follows that $X\setminus X_1$ is open in $X$, which contains $C\setminus X_1$. 
Therefore $C\subseteq (X\setminus X_1)\cup\bigcup\limits_{j\in J} V_j$ is an open cover of $C$. 
Since $C$ is quasi-compact, there exists a finite subset $J_0\subseteq J$ with $C\subseteq (X\setminus X_1)\cup\bigcup\limits_{j\in J_0} V_j$. 
This implies 
$$C\cap X_1\subseteq \bigcup\limits_{j\in J_0} V_j\cap X_1=
\bigcup\limits_{j\in J_0} W_j$$
and we are done.
\end{proof}

\begin{proposition}\label{qc3}
Let $G$ be a Heisenberg group with its unitary dual $X=\widehat{G}$. 
Denote by $\Gamma_2\subseteq X$ the set of characters of $G$, 
and $\Gamma_1:=X\setminus \Gamma_1\simeq\RR^\times$. 
Then a subset $C\subseteq X$ is quasi-compact if and only if it satisfies the following conditions: 
\begin{enumerate} 
	\item $C\cap \Gamma_2$ is a compact subset of $\Gamma_2$;  \item $C\cap\Gamma_1$ is a bounded closed subset of $\Gamma_1$; 
	\item if $0\in\RR$ is an accumulation point of $C\cap\Gamma_1$, then $C\cap \Gamma_2\ne\emptyset$. 
	\end{enumerate}
\end{proposition}

\begin{proof}
We begin by some preliminary remarks on the topology of $X$:  
A subset $F\subseteq X$ is closed if and only if $F\cap \Gamma_j$ is closed in $\Gamma_j$ for $j=1,2$ and 
additionally, if $0\in\RR$ is an accumulation point of $F\cap\Gamma_1$, then $\Gamma_2\subseteq F$. 
Equivalently, a subset $D\subseteq X$ is open if and only if $D\cap\Gamma_j$ is open in $\Gamma_j$ for $j=1,2$ and additionally, 
if $0\in\RR$ is an accumulation point of $\Gamma_1\setminus D$, then $D\subseteq \Gamma_1$. 
This is further equivalent to the fact that 
$D\subseteq X$ is open if and only if $D\cap\Gamma_j$ is open in $\Gamma_j$ for $j=1,2$ and additionally, 
if $D\cap\Gamma_2\ne\emptyset$, then 
$\{t\in\Gamma_1\mid \vert t\vert<\varepsilon_0\}\subseteq D$ 
for suitable $\varepsilon_0>0$. 

Now assume that $C\subseteq X$ is quasi-compact. 
Then $C\cap \Gamma_2$ is a compact subset of~$\Gamma_2$ by Lemma~\ref{qc2}, since the relative topology of $\Gamma_2$ is Hausdorff. 
Moreover, $C\cap\Gamma_1$ is a closed subset of $\Gamma_1$ by Lemma~\ref{qc1}, since the points of $\Gamma_1$ are separated in $X$. 
We will now also check that $C\cap\Gamma_1$ is a bounded subset of  $\Gamma_1\simeq\RR^\times$. 
To this end, for every $n\ge 1$ define the open subset 
$D_n:=\Gamma_2\cup \{t\in \Gamma_1\mid \vert t\vert<n \}\subseteq X$. 
One has $D_1\subseteq D_2\subseteq \cdots$ and $\bigcup\limits_{n\ge 1}D_n=X\supseteq C$. 
Since $C$ is quasi-compact, it then follows that $C\subseteq D_{n_0}$ for some $n_0\ge 1$, that is, $D\cap\Gamma_1$ is bounded in $\Gamma_1$. 
To check the third condition in the statement, 
let us assume that $C\cap\Gamma_2=\emptyset$, that is, $C\subseteq\Gamma_1$. 
Now define the open subset $U_n:=\{t\in\Gamma_1\mid\vert t\vert>1/n\}\subseteq X$ for every $n\ge 1$. 
One has $U_1\subseteq U_2\subseteq \cdots$ and $\bigcup\limits_{n\ge 1}U_n=\Gamma_1\supseteq C$. 
Since $C$ is quasi-compact, it then follows that $C\subseteq U_{n_0}$ for some $n_0\ge 1$, that is,
$0\in\RR$ is not  an accumulation point of $C=C\cap\Gamma_1$. 

Conversely, let $C\subseteq X$ satisfy the three conditions from the statement. 
To prove that $C$ is quasi-compact let $\{W_j\}_{j\in J}$ be an arbitrary open cover of $C$. 
Then 
$$(C\cap\Gamma_1)\sqcup(C\cap\Gamma_2)=C\subseteq\bigcup_{j\in J}W_j.$$
As $\Gamma_2$ is a closed subset of $X$ whose relative topology is Hausdorff, it is straightforward to check that $C\cap \Gamma_2$ is a compact subset of $X$, 
hence there exists a finite subset $J_2\subseteq J$ with 
$C\cap\Gamma_2\subseteq\bigcup\limits_{j\in J_2}W_j$. 
We will now discuss separately two cases that can occur: 

Case 1: $0\in\RR$ is an accumulation point of $C\cap\Gamma_1$. 
Then $C\cap\Gamma_2\ne\emptyset$ by hypothesis, hence there exists $j_0\in J$ with $W_{j_0}\cap\Gamma_2\ne\emptyset$. 
By the preliminary remarks on the topology of $X$, there exists $t_0>0$ with $\{t\in\Gamma_1\mid \vert t\vert<t_0\}\subseteq W_{j_0}$. 
Since $C\cap\Gamma_1$ is a closed bounded subset of $\Gamma_1$, 
it follows that $(C\cap\Gamma_1)\setminus W_{j_0}$ is a compact subset of $C\cap\Gamma_1$, hence there exists a finite subset $J_1\subseteq J$ with 
$(C\cap\Gamma_1)\setminus W_{j_0}\subseteq\bigcup\limits_{j\in J_1}W_j$. 
Then 
$$C\subseteq ((C\cap\Gamma_1)\setminus W_{j_0})
\cup W_{j_0}
\cup (C\cap\Gamma_2)\subseteq \bigcup_{j\in J_1\cup\{j_0\}\cup J_2}W_j$$
hence we have obtained a finite subfamily of $\{W_j\}_{j\in J}$ that still covers $C$. 

Case 2: $0\in\RR$ is not an accumulation point of $C\cap\Gamma_1$. 
Then there exists $t_1>0$ with $C\cap\Gamma_1\subseteq 
\{t\in\Gamma_1\mid \vert t\vert\ge t_1\}$. 
As $C\cap\Gamma_1$ is a closed bounded subset of $\Gamma_1$, it then follows that $C\cap\Gamma_1$ is compact, 
hence there exists a finite subset $J_1\subseteq J$ with 
$C\cap\Gamma_1\subseteq\bigcup\limits_{j\in J_1}W_j$. 
Therefore 
$$C\subseteq \bigcup_{j\in J_1\cup J_2}W_j$$
and we have again obtained a finite subfamily of $\{W_j\}_{j\in J}$ that covers $C$. 
This completes the proof. 
\end{proof}

\begin{corollary}\label{qc4}
If $G$ is a Heisenberg group then there exist two quasi-compact subsets of $\widehat{G}$ whose intersection fails to be quasi-compact.
\end{corollary}

\begin{proof}
	Using the notation of Proposition~\ref{qc3},  
	let $K_2,K_2'\subseteq\Gamma_2$ be any nonempty compact subsets with $K_2\cap K_2'=\emptyset$. 
	For any closed bounded subset $K_1\subseteq\Gamma_1$ for which $0\in\RR$ is an accumulation point, define 
	$C:=K_1\cup K_2$ and $C':=K_1\cup K_2'$. 
	It then follows by Proposition~\ref{qc3} that $C$ and $C'$ are quasi-compact subsets of $X$ for which $C\cap C'$ fails to be quasi-compact. 
\end{proof}

\begin{corollary}\label{corHeis}
	If $G$ is a Heisenberg group then the set of Dini functions on $\widehat{G}$ fails to be closed to pointwise sum, product, and minimum. 
\end{corollary}

\begin{proof}
We recall that since $C^*(G)$ is separable and liminal, the quasi-compact subsets of its spectrum $\widehat{G}$ are $G_\delta$-subsets. (See \cite[4.7.10]{Dix64}.)
Then use Theorem~\ref{Dini1} above along with \cite[Th. 4.12]{Ki04}.
\end{proof}

To put Corollary~\ref{corHeis} into perspective, we recall that if $X$ is locally compact space (that is, also Hausdorff), then the set of all Dini functions on $X$ is exactly the cone $\Cc_0^+(X)$ of all nonnegative functions on $X$ that vanish to infinity, and $\Cc_0^+(X)$
is closed pointwise sum, product, and minimum. 

\section{Noncommutative dimension theory}\label{section6}

In this section we discuss two noncommutative generalizations of the covering dimension of topological spaces. 
Both notions have been proved to be very effective in the study of special classes of $C^*$-algebras, but they were very little studied in the case of $C^*$-algebras of nilpotent Lie groups. 

\subsection{Real rank} 
We start our discussion with the notion of real rank, which in the case of $C^*$-algebras of Lie groups is easier to compute exactly in terms of the corresponding Lie algebra. 

\begin{definition}
	\normalfont
	Let $\Ac$ be any unital $C^*$-algebra, 
	and for any integer $n\ge1$ denote by $\Lc_n(\Ac)$ the set of all $n$-tuples 
	$(a_1,\dots,a_n)\in\Ac^n$ with $\Ac a_1+\cdots+\Ac a_n=\Ac$.  
	We also denote $\Ac^{\sa}:=\{a\in\Ac\mid a=a^*\}$. 
	
	The \emph{stable rank} of $\Ac$ is defined by 
	$$\tsr(\Ac):=\min\{n\ge 1\mid \text{$\Lc_n(\Ac)$ is dense in $\Ac^n$}\}$$ 
	with the usual convention $\min\emptyset=\infty$. 
	The \emph{real rank} of $\Ac$ is similarly defined by 
	$$\RRa(\Ac):=\min\{n\ge 0\mid \text{$\Lc_{n+1}(\Ac)\cap(\Ac^{\sa})^{n+1}$ is dense in $(\Ac^{\sa})^{n+1}$}\}.$$ 
	For any non-unital $C^*$-algebra, its real rank and its stable rank are defined as 
	the real rank, respectively the stable rank, of its unitization. 
\end{definition}

\begin{theorem}\label{th_exp}
	For every connected simply connected nilpotent Lie group $G$ with its Lie algebra $\gg$, 
	we have 
	$$\RRa(C^*(G))=\dim(\gg/[\gg,\gg]).$$ 
\end{theorem}

\begin{proof}
See \cite[Th. 3.5]{BB16b}.
\end{proof}

\begin{corollary}\label{cor_exp}
If $G_1$ and $G_2$ are connected simply connected nilpotent Lie groups whose $C^*$-algebras are $*$-isomorphic, then $\dim(\gg_1/[\gg_1,\gg_1])=\dim(\gg_2/[\gg_2,\gg_2])$, where $\gg_j$ is the Lie algebra of $G_j$ for $j=1,2$.
\end{corollary}

\begin{proof} 
Two $*$-isomorphic $C^*$-algebras have the same real rank, hence the assertion follows by Theorem~\ref{th_exp}. 
\end{proof}

\begin{theorem}\label{th_exp_tsr}
	For every connected simply connected nilpotent Lie group $G$ with its Lie algebra $\gg$, 
	if we denote $r:=\dim(\gg/[\gg,\gg])$, then  
	$$\tsr(C^*(G))= 
	\begin{cases}
	1 &\text{ if and only if }G=\RR, \\
	1+\max\{[r/2],1\} &\text{ otherwise}.
	\end{cases}
	$$ 
\end{theorem}

\begin{proof}
See \cite[Th. 5.4]{BB16b}.
\end{proof}

We note that the above formulas for the real rank and for the stable rank actually hold true unchanged in the more general case of exponential solvable Lie groups, as shown in \cite{BB16b}. 

\subsection{Nuclear dimension}

The notion of nuclear dimension was introduced in  \cite{WZ10} as a noncommutative generalization of the the topological  covering dimension, replacing coverings of topological spaces by appropriate finite-rank approximations of the identity map of a $C^*$-algebra. 

\begin{definition}\label{dimn}
	A $C^*$-algebra $\Ac$  has nuclear dimension less than  $n$ if there is a net $(F_\iota, \phi_\iota, \varphi_\iota)_{\iota\in I}$ of triplets where $F_\iota$ are finite dimensional $C^*$-algebras, and $\varphi_\iota\colon F_\iota\to \Ac$ and $\psi_\iota\colon \Ac\to F_\iota$ are completely positive maps satisfying
	\begin{itemize}
		\item[i) ] $\varphi_{\iota}\circ \psi_{\iota}(a) \to a$ for every $a\in\Ac$; 
		\item[ii)] $\Vert\phi_\iota\Vert \le 1$;
		\item[iii)]  There is a decomposition $F_\iota =\bigoplus\limits_{k=1}^{n+1} F_\iota^{( k)}$, where
		 $F_\iota^{( k)}$ are ideals such that $\varphi_{\iota}\vert_{F_\iota^{(k)}}$ is completely positive and contractive of order zero, that is, whenever $h_1, h_2\in F_\iota^{(k)}$ are positive elements  such that $h_1 h_2=0$, one has $$\varphi_\iota(h_1)\varphi_\iota(h_2)=0.$$
	\end{itemize}
The \textit{nuclear dimension}, denoted $\dim_n\Ac$,  is the smallest 
$n\in \NN\cup \{ \infty\}$ with the above properties. 
\end{definition}

Here are some of the properties of the nuclear dimension that we need later on. 
\begin{enumerate}
\item	 
If $\Ac$ is a separable continuous trace $C^*$-algebra, then 
\begin{equation}\label{nucdim_trace}
	 \dim_n\Ac=\dim (\widehat{\Ac}).
	\end{equation}
	(See \cite[Cor.~2.10]{WZ10}.)
	\item
	For an exact sequence of $C^*$-algebras
	$$0\to \Jc\to \Ac\to \Qc\to 0 $$
	one has 
	\begin{equation}\label{nucdim_exact}
	\max (\dim_n \Jc, \dim_n \Qc) \le \dim_n\Ac \le \dim_n\Jc +\dim_n\Qc +1.
	\end{equation}
	(See \cite[Prop.~2.9]{WZ10}.)
\end{enumerate}

\begin{lemma}\label{dimn_lemma}
	Let $\Ac$ be any separable $C^*$-algebra with a family of closed two-sided ideals 
	$$\{0\}=\Jc_0\subseteq\Jc_1\subseteq\cdots\subseteq\Jc_n=\Ac$$
	where for each $j=1,\dots,n$,   $\Jc_j/\Jc_{j-1}$ has continuous trace, its spectrum $\Gamma_j=\widehat{J_j/J_{j-1}}$ has finite covering dimension, and all its irreducible representations are infinite dimensional.   
	Then 
	$$\max\{\dim (\Gamma_j)\mid j=1,\dots,n\} \le  \dim_n(\Ac)\le 
	\sum\limits_{j=1}^n \dim (\Gamma_j)+n-1.$$ 
	 \end{lemma}

\begin{proof}
	We proceed by induction on $n$. 
	The case $n=1$ is obvious. 
	
	If we assume $n\ge 2$ and the assertion already proved for $n-1$, 
	then, using the family of 
	closed two-sided ideals 
	$$\{0\}=\Jc_1/\Jc_1\subseteq\Jc_2/\Jc_1\subseteq\cdots\subseteq\Jc_n/\Jc_1=\Ac/\Jc_1$$
	of $\Ac/\Jc_1$ for which $(\Jc_j/\Jc_1)/(\Jc_{j-1}/\Jc_1)\simeq \Jc_j/\Jc_{j-1}$ is a continuous trace  $C^*$-algebra,  $j=2,\dots,n$, 
	then the induction hypothesis implies 
$$
\max\{\dim (\Gamma_j)\mid j=2,\dots,n\} \le  \dim_n(\Ac/\Jc_1 )\le 
	\sum\limits_{j=2}^n \dim (\Gamma_j)+n-2.
	$$ 
		Using  \eqref{nucdim_exact} for the exact sequence
	$0\to \Jc_1\to \Ac\to \Ac/\Jc_1\to 0$
	we directly obtain the assertion for $n$, and this completes the proof. 
\end{proof}

In the next proposition we use notation introduced in Section~\ref{section-nilpotent}.

\begin{proposition}\label{dimn_prop}
Let $G$ be a connected simply connected nilpotent Lie group. 
Then for any Jordan-H\"older basis for $\gg$ with corresponding set of jump indices $\{e_j\}_{j=1, \dots, m}$, and coarse partition
$\gg^*/G= \bigcup\limits_{j=1}^m \Xi_{e_j}$, we have
\begin{equation}\label{dimn_ineq}
 \max\{\dim(\Xi_{e_j})\mid j=1, \dots, m\} \le \dim_n C^*(G) \le \sum\limits_{j=1}^m \dim(\Xi_{e_j})+m-1.
 \end{equation}
\end{proposition}

\begin{proof}
	The proposition is a direct consequence of above Lemmas~\ref{dimn_lemma} and \ref{coarse_lemma}.
	\end{proof}

We derive now a rough estimate for the nuclear dimension of the $C^*$-algebra of a nilpotent Lie group. 

\begin{corollary}\label{dimn_cor}
	Let $G$ be a connected simply connected non-abelian nilpotent Lie group, with 
its Lie algebra of dimension $n \ge 1$.
	Then 
	$$  2\le \dim (\gg/[\gg, \gg]) \le \dim_n C^*(G) \le n+ m-1$$
	where $m$ is the coarse length of $G$.
	\end{corollary}
 
 \begin{proof}
 	The set of characters $\Xi_m$ is homeomorphic to the vector space 
 	of functionals in $\gg^*$  vanishing on $[\gg, \gg]$, which has dimension at least $2$ when $G$ is not abelian. 
 	Using this in the first inequality in \eqref{dimn_ineq} proves the first inequality in the statement. 
 	The second inequality follows since each $\Xi_{e_j}$ is homeomorphic to a locally closed semi-algebraic subset of $\gg^*$, hence is it easy to see that $\sum\limits_{j=1}^m \dim(\Xi_{e_j})\le \dim(\gg^*) =n$. 
 	\end{proof}

In particular, the above corollary shows that the $C^*$-algebra of  a nilpotent Lie group~$G$ is never an $AF$-algebra, although, being separable and liminal, it is embeddable in an $AF$-algebra.

\section{Further remarks and examples} \label{section7}
In this final section we illustrate some of the above ideas by some more examples of  connected simply connected nilpotent Lie groups. 

So far in this paper we have encountered the following isomorphism invariants of $\gg$ that can be in principle computed in terms of $C^*(G)$ or its space of primitive ideals $\Prim(C^*(G))\simeq\widehat{G}$: 
\begin{itemize}
	\item $a(G):=\dim(\gg/[\gg,\gg])$, which is equal to the real rank of $C^*(G)$ (Theorem~\ref{th_exp}); 
	\item $\ind G:=\ind \widehat{G}$ (see Remark~\ref{ind_invar} and Theorem~\ref{dual_spec}); 
	\item $\lgth G:=\lgth\widehat{G}$ (see Definition~\ref{solvspecl_space} and  Theorem~\ref{dual_spec}). 
\end{itemize}

To facilitate the applications of these invariants, we recall in Proposition~\ref{ind_prop} below how the index $\ind G$ can be directly computed in terms of the structure of the Lie algebra~$\gg$.

\begin{proposition}\label{ind_prop}
	Let $G$ be any connected simply connected nilpotent Lie group with its Lie algebra $\gg$ and for every $\xi \in \gg^*$ define 
	$B_\xi \colon \gg\times \gg \to \RR$, 
	$B_\xi(X, Y) =\langle \xi, [X, Y]\rangle$. 
	Then the following assertions hold: 
	\begin{enumerate}
		\item\label{ind_prop_item1} 
		If $e_1\prec\cdots\prec e_n\prec\emptyset$ are all the index sets of the coadjoint orbits of $G$ 
		with respect to some Jordan-H\"older basis of $\gg$, then 
		$$
		\begin{aligned}
		\ind G & =\dim\gg-\card\, e_1\\
	&	=\dim\gg-\max_{\xi\in\gg^*}\dim\Ad^*_G(G)\xi \\
	&	=\min_{\xi\in\gg^*}\dim G(\xi)\\
	& = \dim \gg-\max_{\xi\in \gg^*}
		\mathrm{rank } B_\xi.
		\end{aligned}$$ 
		\item\label{ind_prop_item2} 
		One has $\ind G=r$ if and only if there exists some open dense subset $V\subseteq\widehat{G}$ 
		that is homeomorphic to some open subset of $\RR^r$. 
	\end{enumerate}
\end{proposition}

\begin{proof}
See \cite[Rem. 4.8 and Prop. 4.9]{BBL17}.
\end{proof}

\begin{example}[stepwise square-integrable representations]
	\label{stepwise1}
	\normalfont
Let $\ng$ be a nilpotent Lie algebra with a fixed Jordan-H\"older basis, 
and a distinguished subset of the corresponding Jordan-H\"older series, denoted  
$\{0\}=\ng_0\subseteq\ng_1\subseteq\cdots\subseteq\ng_q=\ng$. 
Assume that for every $j=1,\dots,q$ one has a subalgebra $\mg_j\subseteq\ng_j$ 
and a linear subspace $\Vc_j\subseteq\mg_j$ compatible with the above fixed basis, 
satisfying the following conditions: 
\begin{itemize}
	\item The center $\zg_j$ of $\mg_j$ is 1-dimensional, $\mg_j=\zg_j\dotplus \Vc_j$, 
	$\mg_j$ has flat generic coadjoint orbits, 
	and one has the semidirect product decomposition $\ng_j=\mg_j\ltimes\ng_{j-1}$. 
	\item One has $[\mg_j,\ng_{j-1}]\subseteq\Vc_1\dotplus\cdots\dotplus\Vc_{j-1}
	$ 
	and $[\mg_j,\zg_1+\cdots+\zg_{j-1}]=\{0\}$. 
\end{itemize}
Let 
$\Xc:=\{\xi\in\ng^*\mid J(\xi\vert_{\ng_j})=e(\ng_j)\text{ for }j=1,\dots,q\}$ 
and $\sg:=\zg_1+\cdots+\zg_q$.
Then one has: 
\begin{enumerate}
	\item The set $\Xc$ is $\Ad^*_N$-invariant. 
	\item For any $N$-coadjoint orbit $\Oc\subseteq\ng^*$ one has 
	$\Oc\subseteq\Xc$ if and only if there exists  
	$\xi\in\Oc$ with $\ng(\xi)=\sg$ 
	and  
	$
	\Vc_1\dotplus\cdots\dotplus\Vc_q\subseteq\Ker\xi$, 
	and then $\xi$ is uniquely determined by these properties. 
	\item If $[\Xc]\subseteq\widehat{N}$ is the subset of the unitary dual of~$N$ corresponding to the $N$-coadjoint orbits contained in~$\Xc$,   
	then the interior of $[\Xc]$ is dense in $\widehat{N}$, and the relative topology of $[\Xc]$ is Hausdorff.   
\end{enumerate}
See \cite[Th. 2.8 and Cor. 2.11]{BB16a} for proofs of these assertions. 

In the above setting one has $\ng=\sg\dotplus(\Vc_1\dotplus\cdots\dotplus\Vc_q)$ and $\dim\sg=q$ hence $\dim\Oc=\dim\ng-q$ for every coadjoint orbit $\Oc\subseteq \Xc$. 
It  thus follows by Proposition~\ref{ind_prop} that $\ind N=q$. 
\end{example}

\begin{example}
	\normalfont
	The above Example~\ref{stepwise1} 
	applies to any nilpotent Lie algebra~$\ng$ 
	that is the nilradical of a minimal parabolic subalgebra  
	of a split real form of a complex semisimple Lie algebra (cf. also \cite{Wo13}), 
	and in this case each $\mg_r$ is a Heisenberg algebra. 
	In particular, if $N_k$ is the group of strictly upper triangular real matrices of size $k\ge 3$, and $q_k$ is the greatest integer that is less than $k/2$, then $q_k$ is equal to the number of entries of such a matrix 
	that belong to the secondary diagonal, hence 
	we obtain $\ind N_k=q_k$. 
\end{example}

\begin{example}\label{stepwise3}
	\normalfont
	If $\ng$ is a 3-step nilpotent Lie algebra with 1-dimensional center denoted by~$\zg$, 
	then by \cite[Th. 5.1--5.2]{BB15} there exists a semidirect product decomposition 
	$\widetilde{\ng}=\ag\ltimes\mg$, 
	where $\ag$ is an abelian Lie algebra, $\mg$ is a 3-step nilpotent Lie algebra with its center equal to~$\zg$ and with generic flat coadjoint orbits, 
	and there exists a linear subspace $\Vc\subseteq\mg$ with $\mg=\zg\dotplus\Vc$, 
	$[\ag,\Vc]\subseteq\Vc$, and $[\ag,\zg]=\{0\}$. 
	Hence by writing $\ag$ as a direct sum of 1-dimensional abelian algebras, one sees that $\ng$ is a Lie algebra to which Example~\ref{stepwise1} applies and we obtain $\ind N=1+\dim\ag$. 
\end{example}

To illustrate the area of applicability of Example~\ref{stepwise3} we now provide 
an uncountable family of pairwise nonisomorphic  3-step nilpotent Lie algebras with 1-dimensional centers. 
Recall that there exist only countably many isomorphism classes of 2-step nilpotent Lie algebras with 1-dimensional centers, 
since these are precisely the Heisenberg algebras, 
hence there exists precisely one isomorphism class for every odd dimension, 
and no isomorphism classes for the even dimensions. 
The following example  
was also discussed in \cite{BBP15} in connection with $L^2$-boundedness properties of the Weyl-Pedersen calculus on nilpotent Lie groups. 

\begin{example}
	\normalfont
	For all $s,t\in\RR\setminus\{0\}$ let $\gg_0(s,t)$ be the 6-dimensional 2-step nilpotent Lie algebra 
	defined by a basis $X_1,X_2,X_3,X_4,X_5,X_6$ 
	with the commutation relations 
	$$[X_6,X_5]=sX_3,\ [X_6,X_4]=(s+t)X_2,\ [X_5,X_4]=tX_1. $$
	It follows by \cite[Ex. 3.5]{La05} and \cite[Ex. 5.2]{La06} that 
	$$\{\gg_0(s,t)\mid s^2+st+t^2=1,\ 0<t\le 1/\sqrt{3}\}$$
	is a family of isomorphic Lie algebras 
	that are however pairwise non-isomorphic as symplectic Lie algebras 
	with 
	the common symplectic structure $\omega$ given by the matrix 
	$$J_{\omega}=
	\begin{pmatrix}
	\hfill 0 &\hfill 0 &\hfill 0 & 0 & 0 & 1 \\
	\hfill 0 &\hfill 0 &\hfill 0 & 0 & 1 & 0 \\
	\hfill 0 &\hfill 0 &\hfill 0 & 1 & 0 & 0 \\
	\hfill 0 &\hfill 0 & -1      & 0 & 0 & 0 \\
	\hfill 0 & -1      &\hfill 0 & 0 & 0 & 0 \\
	-1 &\hfill 0 &\hfill 0 & 0 & 0 & 0 
	\end{pmatrix}$$
	If we define $\gg(s,t):=\RR\dotplus_\omega\gg_0(s,t)$ 
	then we obtain a family of 3-step nilpotent Lie algebras with 1-dimensional centers. 
	These Lie algebras are moreover pairwise nonisomorphic. (See for instance \cite[Ex. 6.8]{BB15}.) 
	If $G(s,t)$ is the connected simply connected nilpotent Lie group associated with $\gg(s,t)$, then we obtain $\ind G(s,t)=1$ by Example~\ref{stepwise3}. 
\end{example}

\appendix

\section{General topology}\label{AppA}
Here we collect some terminology on (generally non-Hausdorff) topological spaces as used in the main body of this paper. 
Let $X$ be topological space with its set of closed subsets denoted by $\Fc(X)$. 
For any subset $A\subseteq X$ we denote its closure by $\overline{A}$. 
\begin{itemize}
	\item $X$ has \emph{property $T_0$} if the map 
$X\to\Fc(X)$, $x\mapsto\overline{\{x\}}$
is injective. 
\item $X$ has \emph{property $T_1$} if 
for all $x\in X$ one has $\overline{\{x\}}=\{x\}$. 
\item $X$ is \emph{almost Hausdorff} if for every closed subset $F\subseteq X$ there exists a subset $A\subseteq F$ 
which is open in the relative topology of $F$ 
and whose relative topology is Hausdorff. 
\item A  point $x\in X$ is said to be \textit{separated  in} $X$ 
if for every $x'\in X\setminus\overline{\{x\}}$ 
there exist open subsets $V, V'\subset X$ with $x\in V$, $x'\in V'$ and $V\cap V' =\emptyset$.
\item $X$ is \emph{quasi-compact} if for every family $\{V_i\}_{j\in J}$ of open subsets of $X$ with $\bigcup\limits_{j\in J} V_j=X$ there exists a finite set $J_0\subseteq J$ with $\bigcup\limits_{j\in J_0} V_j=X$. 
\item $X$ is \emph{locally quasi-compact} if every point of $X$ has a base of quasi-compact neighborhoods. 
\item $X$ is a \emph{Baire space} if the intersection of every countable family of dense open subsets of $X$ is dense in $X$. 
\item $X$ is a \emph{spectral space} if 
the image of the map $X\to\Fc(X)$, $x\mapsto\overline{\{x\}}$ is equal to the set of all $F\in\Fc(X)$ with the property that if $F\subseteq F_1\cup F_2$ with $F_1,F_2\in\Fc(X)$, then $F\subseteq F_1$ or $F\subseteq F_2$. 
\item $X$ is a \emph{Polish space} if it is separable and its topology is complete metrizable. 
\item A subset $A\subseteq X$ is \emph{locally closed} if $A=F_1\setminus F_2$ for suitable $F_1,F_2\in\Fc(X)$. 
\item A  bounded, lower semicontinuous function $f\colon X\to[0,\infty)$  is a 
\emph{Dini function} if it satisfies  
the following condition: 
For every upwards directed net of lower semicontinuous functions $\{f_j\colon X\to[0,\infty)\}_{j\in J}$ 
with $f=\sup\limits_{j\in J}f_j$ pointwise on $X$,  
one has $f=\lim\limits_{j\in J}f_j$ uniformly on $X$.  
\item $X$ is a \emph{Dini space} if it has the property~$T_0$, is spectral, is second countable,  
and its topology has a base consisting of the sets $f^{-1}(0,\infty)$ for all Dini functions $f$ on $X$. 
\end{itemize}

\subsection*{Acknowledgment} 
This work was supported by a grant of the Romanian National Authority for Scientific Research and Innovation, CNCS--UEFISCDI, project number PN-II-RU-TE-2014-4-0370.

\end{document}